\documentclass[10pt,reqno]{amsart}
\usepackage{amssymb,amsfonts}
\usepackage[all,arc]{xy}
\usepackage{enumerate}
\usepackage{mathrsfs}
\usepackage{geometry}
\usepackage{amsmath}
\usepackage{latexsym, bm}
\usepackage{indentfirst}
\usepackage{CJK}
\usepackage{amsthm}
\usepackage{hyperref}
\usepackage{cite}
\usepackage{lipsum}

\newtheorem{thm}{Theorem}[section]
\newtheorem{prop}[thm]{Proposition}
\newtheorem{lem}[thm]{Lemma}
\newtheorem{cor}[thm]{Corollary}
\newtheorem{conj}[thm]{Conjecture}

\theoremstyle{definition}
\newtheorem{defn}[thm]{Definition}
\newtheorem{remk}[thm]{Remark}
\newtheorem{exam}[thm]{Example}

\begin{document}
\title{$\hat{H}$-eigenvalues of Hermitian tensors and some applications}
\author{Haojie Chen, Yang Yang}
\address{Department of Mathematics, Zhejiang Normal University, Jinhua Zhejiang, 321004, China}
\email{chj@zjnu.edu.cn; yy20163243@163.com}
\date{}
\maketitle

\begin{abstract}
We introduce $\hat{H}$-eigenvalue for $2m$-th order $n$-dimensional complex tensors. Then we determine several checkable inclusion sets for $\hat{H}$-eigenvalues and derive some criterions for the Hermitian positive definiteness (semi-definiteness) of Hermitian and CPS tensors. We also apply the Hermitian tensors to study holomorphic sectional curvature in complex differential geometry and reprove the algebraic part of recent results by Alvarez-Heier-Zheng \cite{AHZ} and Chaturvedi-Heier \cite{CH}.
\end{abstract}

\tableofcontents
\section{Introduction}

The theory of tensors, as a generalization of matrix theory has been widely applied in pure and applied mathematics. An intrinsic feature of tensors is the spectral theory of tensors, i.e. tensor eigenvalues. The eigenvalues of higher-order tensors were introduced independently by Qi \cite{Q1} and Lim \cite{Lim} in 2005. Since then, the spectral theory of tensors has been extensively studied and has been found a lot of applications in many subjects such as statistical data analysis, higher order Markov chains, polynomial optimization theory, hypergraph partition and so on (see \cite{QL} and the reference therein). 

The eigenvalues of tensors are closely related to the notion of positive definite (PD) tensors. The following result is proved in \cite{Q1} which starts the theory.
\begin{thm} [Qi] \label{1.1}
Let $A$ be a real symmetric $2m$-th order $n$-dimensional tensor. Then $A$ is positive definite (semidefinite) if and only if its H-eigenvalues are all positive (nonnegative). 
\end{thm}
 An eigenvalue inclusion set of a tensor is a set where its H-eigenvalues lie in. Many important properties about tensors can be obtained by studying the eigenvalue inclusion sets. A fundamental problem is to find different checkable inclusion sets for a given tensor. Qi \cite{Q1} first gave an eigenvalue inclusion set for real symmetric tensors, which is a generalization of the well-known Ger\u{s}gorin inclusion set of matrices. After that, several new eigenvalue inclusion sets of real tensors
have been found. Among them, two notable inclusion sets are given by Li-Li \cite{LL} and Li-Li-Kong \cite{LLK}. They constructed two Brauer type
eigenvalue inclusion sets for general tensors and showed that they are tighter than the set given by Qi. For more results on different types of eigenvalue inclusion sets, we refer to \cite{BWS}\cite{LCL}\cite{LJL}\cite{LZL}\cite{WZC}.

Symmetric real tensors is in one-to-one correspondence with homogeneous polynomials in several real variables. For homogeneous polynomials in complex conjugate variables which take real values, it is the conjugate partial symmetric (CPS) tensors which play a similar role. Jiang-Li-Zhang \cite{JLZ} first introduced the concept of CPS tensors and showed that if a homogeneous multivariate complex polynomial takes real values, it is induced by a unique CPS tensor. They also studied other properties of CPS tensor involving their C-eigenvalues. Ni \cite{N} introduced the concept of Hermitian tensors which is more general than CPS tensors and applied it to quantum physics. A CPS tensor is a symmetric Hermitian tensor of the same dimension. As the higher dimensional generalizations of Hermitian matrices, CPS tensors and Hermitian tensors are expected to play an important role in the theory of tensors. We refer to \cite{FJL1}\cite{FJL2}\cite{NY} for recent studies on CPS tensor and Hermitian tensors. 

In this paper, we mainly study the spectral properties of Hermitian tensors. First, we have the following definition of Hermitian positive definiteness of a $2m$-th order complex tensor, as a generalization of positive definite Hermitian matrices. \footnote{Our definition of Hermitian positive definiteness differs from the definition of positive Hermitian tensors in \cite{N} that we input the same vector in (\ref{1.12}) hence is weaker.}
\begin{defn} \label{complex1}
Let $A=(a_{i_1\cdots i_m\bar{j}_1\cdots \bar{j}_m})$ be a $2m$-th order $n$-dimensional complex tensor. $A$ is said to be Hermitian positive definite (semidefinite) if for any nonzero complex vector $\mathbf{x}=(x_1,\cdots, x_n)^T\in \mathbb C^n\setminus \{0\}$, \begin{align} f(A)(\mathbf{x})=\sum_{i_1,\cdots, i_m, j_1, \cdots, j_m=1}^n a_{i_1\cdots i_m\bar{j}_1\cdots \bar{j}_m} x_{i_1}\cdots x_{i_m}\overline{x_{j_1}}\cdots\overline{x_{j_m}}>0 (\geq 0).\label{1.12}\end{align}
\end{defn}

When $A$ is a real tensor, the condition that (\ref{1.12}) holds only for real nonzero vectors $\mathbf{x}$ gives the usual definition of positive definiteness of $A$ (\cite{QL}). If $m=1$ and $A$ is an $(n\times n)$ real symmetric matrix, then Hermitian positive definiteness is equivalent to positive definiteness. When $m>1$, we show that Hermitian positive definiteness is in general strictly stronger than positive definiteness for real tensors (see section 2). Therefore, the positivity of H-eigenvalues no longer implies Hermitian positive definiteness. To overcome this, we introduce the following concepts of $\hat{H}$-eigenvalues for $2m$-th order complex tenors.
\begin{defn}
Let $A=(a_{i_1\cdots i_m\bar{j}_1\cdots \bar{j}_m})$ be a $2m$-th order $n$-dimensional complex tensor. We call that $\lambda\in \mathbb C$ is an $\hat{H}$-eigenvalue of $A$ with $\hat{H}$-eigenvector $\mathbf{x}$ if there is a nonzero complex vector $\mathbf{x}=(x_1,\cdots, x_n)^T\in \mathbb C^n\setminus \{0\}$ such that $$\sum_{i_2,\cdots, i_m, j_1, \cdots, j_m=1}^n a_{ii_2\cdots i_m\bar{j}_1\cdots \bar{j}_m} x_{i_2}\cdots x_{j_m}\overline{x_{j_1}}\cdots\overline{x_{j_m}}=\lambda \bar{x}_{i}|x_{i}|^{2m-2}, {\forall}\  1\leq i\leq n.$$ 
	\end{defn}
 If $m=1$,  $\hat{H}$-eigenvalues identify with ordinary matrix eigenvalues. We prove the following result.
\begin{thm} \label{thm1.2}
Let $A$ be a $2m$-th order $n$-dimensional CPS tensor. Then $A$ is Hermitian positive definite (semidefinite) if and only if the $\hat{H}$-eigenvalues of $A$ are all positive (nonnegative). 
\end{thm}

 With this characterization, $\hat{H}$-eigenvalue provides an effective tool to study Hermitian positive definiteness of Hermitian tensors. We study inclusion sets of $\hat{H}$-eigenvalues first. Along the line of real tensors, we determine several checkable inclusion sets for the $\hat{H}$-eigenvalues of $A$, including the Ger\u{s}gorin type inclusion set $K_{ger}(A)$, the Li-Li-Kong type inclusion set $K_{llk}(A)$ and Li-Li type inclusion set $K_{ll}(A)$. We refer to Definition \ref{d4.1}, Definition \ref{d4.3}, Definition \ref{d4.5} for the precise definitions.
Denote the set of all $\hat{H}$-eigenvalues of a complex tensor $A$ to be $\sigma_h(A)$. We prove the following result in section 4.
\begin{prop} Let $A$ be a $2m$-th order $n$-dimensional complex tensor. Then $K_{ger}(A)$, $K_{llk}(A)$ and $K_{ll}(A)$ are all inclusion set for the  $\hat{H}$-eigenvalues of $A$. More precisely, we have 
	$\sigma_h(A)\subseteq K_{ll}(A)\subseteq K_{llk}(A)\subseteq K_{ger}(A).$ 
\end{prop}
Explicit matrix examples are given in section 4 which illustrate that the above inclusion are strict. We remark that several other inclusion sets for H-eigenvalues can be also defined similarly for $\hat{H}$-eigenvalues of Hermitian tensors. 

Applying these inclusion sets, we derive some criterions for the Hermitian positive definiteness (semi-definiteness) of Hermitian tensors. In particular, we define diagonally dominated Hermitian tensor as follows.
\begin{defn}
	Let $A=(a_{i_1\cdots i_m\bar{j}_1\cdots \bar{j}_m})$ be a $2m$-th order $n$-dimensional Hermitian tensor. Denote $r_i(A)=\sum_{(i_2\cdots i_m\bar{j}_1\cdots \bar{j}_{m})\neq (i\cdots i\bar{i}\cdots\bar{i})}|a_{ii_2\cdots i_{m}\bar{j}_1\cdots \bar{j}_m}|$. $A$ is called to be (strictly) diagonally dominated if for any $1\leq i\leq n$, $a_{i\cdots i \bar{i}\cdots \bar{i}}\geq r_i(A)$ ($a_{i\cdots i \bar{i}\cdots \bar{i}}> r_i(A))$. \end{defn}
Using $\hat{H}$-eigenvalues, we prove that (Proposition 4.5) 
	\begin{prop} \label{p1.7}
	If $A$ is (strictly) diagonally dominated, then $A$ is Hermitian positive (definite) semidefinite.
	\end{prop}
 We also define the LLK type and LL type Hermitian tensor in section 4 and show that they are more general than diagonally dominated tensors and are both Hermitian positive semidefinite. An approximate criterion of Hermitian positive definiteness (Proposition \ref{p4.16}) is obtained in section 4.

In the last part of the paper, we relate the Hermitian tensors to holomorphic sectional curvature in complex differential geometry. Holomorphic sectional curvature is a fundamental invariant on a Hermitian manifold, analogous to the sectional curvature on a Riemannian manifold. One central topic in complex geometry is to study compact Hermitian manifolds with positive or negative holomorphic sectional curvature. Many significant progresses have been made recently. In particular, Wu-Yau \cite{WY1} proves that projective manifolds with negative holomorphic sectional curvature must have ample canonical bundle (it is later extended by Tosatti-Yang \cite{TY} to the K\"ahler setting, see also \cite{DT}\cite{WY2}\cite{YZ} for more generalization). In the positive curvature side, there is a famous conjecture by Yau.

\begin{conj} [Yau, Problems 67 in \cite{Yau}] Consider a compact K\"ahler manifold with positive holomorphic sectional curvature, is it
unirational? Is it projective? If a projective manifold is obtained by blowing
up a compact manifold with positive holomorphic sectional curvature along
a subvariety, does it still carry a metric with positive holomorphic sectional
curvature?
\end{conj}

In \cite{Y}, Yang proves that compact K\"ahler manifolds with positive holomorphic sectional curvature must be projective and rationally connected, hence solves a majority of the above Yau's conjecture. The remaining part of the conjecture, namely, to find metrics with positive holomorphic sectional
curvature on the blow up of a projective manifold is still open. There are some progress towards this problem recently. In \cite{AHZ}, Alvarez-Heier-Zheng prove that the projectivization of any holomorphic vector bundle over a compact K\"ahler manifold with positive holomorphic sectional curvature admits a K\"ahler metric with positive holomorphic sectional curvature. In \cite{CH}, Chaturvedi-Heier prove that the total space of a compact holomorphic fibration
admits Hermitian metrics of positive holomorphic sectional curvature if its fiber and base do so. In general, it would be very interesting if more examples of Hermitian metrics with positive holomorphic sectional curvature can be constructed. In section 5, we get that the coefficient tensor of holomorphic sectional curvature is a 4-th order Hermitian tensor. This builds a bridge between complex differential geometry and the spectral theory of tensors. We then apply the criterions for Hermitian positive definiteness to reprove the algebraic part of the results by Alvarez-Heier-Zheng and Chaturvedi-Heier. 
We expect that there are further applications of the $\hat{H}$-eigenvalue in studying problems in differential geometry.

The structure of the paper is as follows. In section 2, we fix some standard notations, following \cite{QL}. In section 3, we introduce the $\hat{H}$-eigenvalues for a $n$-dimensional Hermitian tensor and show its relation with the Hermitian positive-definiteness. In section 4, we derive the $\hat{H}$-eigenvalues inclusion sets of a Hermitian tensor and introduce some Hermitian positive definite tensors. In section 5, we relate the $\hat{H}$-eigenvalues to holomorphic sectional curvature in complex differential geometry.\\

\noindent\textbf{Acknowledgements.}  The first author is supported by National Natural Science Foundation of
China (11901530) and Zhejiang Provincial Excellent Youth Science Foundation (LZYQ25A010001). The authors would like to thank Professors Lei Ni, Xiaolan Nie and Fangyang Zheng for many helpful discussions.

\section{Real and Hermitian tensors}
In this section, we give some preliminary on tensors. We refer to \cite{QL} for the terminology and notations.
\begin{defn} A $m$-th order complex tensor $A$ is a multi-array of entries $A=(a_{i_1\cdots i_m})$ with $a_{i_1\cdots i_k}\in \mathbb C$ for any $1\leq i_1\leq n_1, \cdots, 1\leq i_k\leq n_m$. 
	\end{defn}
A second order tensor is just a complex matrix. If $n_1=n_2=\cdots=n_m=n$, then $A$ is called an $n$-dimensional tensor, or a tensor of the same dimension $n$. $A$ is called a real tensor if $a_{i_1\cdots i_m}$ are all real numbers. For a real tensor of the same dimension $n$, the following concepts of definiteness are defined in \cite{QL}.
\begin{defn} \label{real}
Let $A=(a_{i_1\cdots i_m})$ be a $m$-th order $n$-dimensional real tensor. $A$ is said to be positive definite (semidefinite) if for any nonzero real vector $\mathbf{x}=(x_1,\cdots, x_n)^T\in \mathbb R^n\setminus \{0\}$, 
\begin{align}
\sum_{i_1,\cdots, i_m=1}^n a_{i_1i_2\cdots i_m} x_{i_1}x_{i_2}\cdots x_{i_m}>0 (\ \geq 0).\label{psd} \end{align}
The notion of negative definiteness (semidefiniteness) is defined analogously. 
\end{defn}
Clearly, a positive or negative definite tensor must be with even order $m$.  A \textbf{symmetric} tensor is a tensor $\mathcal B=(b_{i_1\cdots i_m})$ of the same dimension such that $$b_{i_1\cdots i_m}=b_{i_{\sigma(1)}\cdots i_{\sigma(m)}},$$ for any $1\leq i_1\cdots i_m\leq n$, where $\sigma\in S_n$ is any permutation of $\{1,\cdots, m\}$. Given any tensor $A$, its symmetrization $Sym(A)=(b_{i_1\cdots i_m})$ is defined to be 
$$b_{i_1\cdots i_m}=\dfrac{1}{n!}\sum_{\sigma\in S_n} a_{i_{\sigma(1)}\cdots i_{\sigma(m)}}.$$ Then $Sym(A)$ is a symmetric tensor. Also, $A$ is positive (negative) definite if and only if 
$Sym(A)$ is positive (negative) definite. Qi \cite{Q1} introduced the concepts of eigenvalue and H-eigenvalue of an $n$-dimensional tensor which are closely related to the property of positive definiteness. 
\begin{defn}  Given an $m$-th order $n$-dimensional tensor $A=(a_{i_1i_2\cdots i_m})$, if there is a nonzero complex vector $\mathbf{x}=(x_1,\cdots, x_n)^T\in \mathbb C^n\setminus \{0\}$ such that for each $1\leq i\leq n$, $$\sum_{i_2,i_3,\cdots, i_m=1}^n a_{ii_2\cdots i_m} x_{i_2}\cdots x_{i_m}=\lambda x_i^{m-1},$$ then $\lambda$ is called an eigenvalue of $A$ with the eigenvector $\mathbf{x}$. If $\mathbf{x}$ is a real vector, then $\lambda$ is called an H-eigenvalue with $H$-eigenvector $\mathbf{x}$. 
\end{defn}
The following result is obtained in \cite{Q1}. 
\begin{thm}[Qi] A symmetric real tensor with even order $m$ is positive definite (semidefinite) if and only if its H-eigenvalues are all positive (nonnegative). 
\end{thm}

There are also other concepts of eigenvalues of tensors in the literature, for example $E$-eigenvalue, $Z$-eigenvalues, $U$-eigenvalues, $US$-eigenvalue, $C$-eigenvalues and $G$-eigenvalues (see \cite{JLZ}\cite{NQB}\cite{N}\cite{Q1}). They are in general different and can be applied to characterise the positive definiteness in different situations. The H-eigenvalue has the advantage that its defining equations are homogeneous and its inclusion set is easily checkable.

  Let $f(\mathbf{x})$ be a multi-variate homogenous degree $m$ polynomials with $\mathbf{x}=(x_1,\cdots, x_n)^T\in \mathbb C^n$. Then there is a unique symmetric $m$-th order tensors $A=(a_{i_1\cdots i_m})$ such that $f(\mathbf{x})=\sum_{i_1,\cdots, i_m=1}^n a_{i_1\cdots i_m} x_{i_1}\cdots x_{i_m}$. In \cite{JLZ},  Jiang-Li-Zhang generalizes this to homogeneous polynomials in complex conjugate variables which takes real values. More precisely, they define
 \begin{defn} 
 A $d$-th degree homogeneous conjugate complex polynomial is a polynomial in the following form
 $$f(\mathbf{x})=\sum_{k=0}^d\sum_{i_1,\cdots, i_k=1}^n\sum_{j_1, \cdots, j_{d-k}=1}^n a_{i_1\cdots i_k \bar{j}_1\cdots \bar{j}_{d-k}} x_{i_1}\cdots x_{i_k}\overline{x_{j_1}}\cdots \overline{x_{j_{d-k}}}$$ for any $\mathbf{x}=(x_1,\cdots, x_n)^T\in \mathbb C^n$. $f(\mathbf{v})$ is said to take real values if $f(\mathbf{v})\in \mathbb R$ for any complex vectors $\mathbf{v}$.
 \end{defn}
 Correspondingly, they introduce the following tensor.
 \begin{defn}
 A $2m$-order $n$-dimensional complex tensor $A=(a_{i_1\cdots i_m\bar{j}_1\cdots \bar{j}_m})$ is called a conjugate partial-symmetric (CPS) if it satisfies the following:
 \begin{align}
\overline{a_{i_1\cdots i_m\bar{j}_1\cdots \bar{j}_m}}&=a_{j_1\cdots j_m\bar{i}_1\cdots \bar{i}_m}\label{sym1} \\
 a_{i_1\cdots i_m\bar{j}_1\cdots \bar{j}_m}&=a_{i_{\sigma(1)}\cdots i_{\sigma(m)}\bar{j}_{\tau(1)}\cdots \bar{j}_{\tau(m)}} \label{symm}
 \end{align}
 for any $1\leq i_1, \cdots, i_m, j_1,\cdots,j_m\leq n$, where $\sigma, \tau$ are any two independent permutations of $\{1,\cdots, m\}$. 
 \end{defn}
\begin{remk}
The above notation of $2m$ order complex tensor is adopted here due to the convention in differential geometry. It has the advantage that the indices of $x_i$ and $\bar{x}_i$ are divided into two groups.
\end{remk}
 Given a $2m$-order $n$-dimensional complex tensor $A=(a_{i_1\cdots i_m\bar{j}_1\cdots \bar{j}_m})$, there is an associated degree $2m$ conjugate multi-variate polynomial $f(A)$ given by
 $$f(A)(\mathbf{x})=\sum_{i_1,\cdots, i_m, j_1, \cdots, j_m=1}^n a_{i_1\cdots i_m\bar{j}_1\cdots \bar{j}_m} x_{i_1}\cdots x_{i_m}\overline{x_{j_1}}\cdots\overline{x_{j_m}}$$ for any $\mathbf{x}=(x_1,\cdots, x_n)^T\in \mathbb C^n$. Jiang-Li-Zhang prove the following result in \cite{JLZ}.
\begin{thm}[Jiang-Li-Zhang] \label{thm2.8} $f(A)$ takes real values if and only if $A$ satisfies the condition (\ref{sym1}). Conversely, if a $2m$ degree homogenous conjugate multi-variate polynomial $f$ takes real values, there is a unique 2m-th order n-dimensional CPS tensors $A$ such that $f=f(A)$.
\end{thm}
In \cite{N}, to apply tensors to study quantum states, Ni introduces the concept of Hermitian tensors. 
\begin{defn} A $2m$-th order complex tensor $A=(a_{i_1\cdots i_m\bar{j}_1\cdots \bar{j}_m})\in \mathbb C^{n_1\times \cdots \times n_m\times n_1\times \cdots \times n_m}$ is called a Hermitian tensor if it satisfies $$\overline{a_{i_1\cdots i_m\bar{j}_1\cdots \bar{j}_m}}=a_{j_1\cdots j_m\bar{i}_1\cdots \bar{i}_m}$$ for any combination $(i_1\cdots i_mj_1\cdots j_m)$. 
When $n_1=\cdots=n_m=n$, it is a called a $2m$-th order $n$-dimensional Hermitian tensor.
\end{defn}
Clearly, a CPS tensor is a Hermitian tensor of the same dimension which satisfies the symmetric condition (\ref{symm}).

To study when a $2m$ degree homogenous real-valued conjugate multi-variate polynomials takes positive (nonnegative) values, we introduce the following definition of Hermitian positive definiteness (semi-definiteness).
\begin{defn} \label{complex}
Let $A=(a_{i_1\cdots i_m\bar{j}_1\cdots \bar{j}_m})$ be a $2m$-th order $n$-dimensional complex tensor. $A$ is said to be Hermitian positive definite (semidefinite) if for any nonzero complex vector $\mathbf{x}=(x_1,\cdots, x_n)^T\in \mathbb C^n\setminus \{0\}$, $$f(A)(\mathbf{x})=\sum_{i_1,\cdots, i_m, j_1, \cdots, j_m=1}^n a_{i_1\cdots i_m\bar{j}_1\cdots \bar{j}_m} x_{i_1}\cdots x_{i_m}\overline{x_{j_1}}\cdots\overline{x_{j_m}}>0 (\geq 0).$$
\end{defn}
By Theorem \ref{thm2.8}, a Hermitian positive definite or semidefinite tensor must be a Hermitian tensor. When $A$ is a real Hermitian tensor, we compare the two notions of positive definiteness.
\begin{prop} \label{diff}
Let $A$ be a $2m$-th order $n$-dimensional real Hermitian tensor. If $A$ is Hermitian positive definite (semidefinite), then $A$ is positive definite (semidefinite). When $m=1$, $A$ is Hermitian positive definite (semidefinite) if and only if it is positive definite (semidefinite).
\end{prop}
\begin{proof}
The first part follows directly from the definitions. When $m=1$, $A=(a_{i\bar{j}})$ with $a_{i\bar{j}}=a_{j\bar{i}}$ being real numbers. For $\mathbf{x}=(x_1,\cdots, x_n)^T\in \mathbb C^n\setminus \{0\}$, 
$$f(A)(\mathbf{x})=\sum_{i,j=1}^n a_{i\bar{j}}x_i\overline{x_{j}}=\sum_{i,j=1}^n (a_{i\bar{j}}Re(x_i)Re(x_{j})+a_{i\bar{j}}Im(x_i)Im(x_{j})).$$
So $f(A)(\mathbf{x})> 0(\geq 0)$ if $A$ is positive definite (semidefinite).
\end{proof}
The following example shows that when $m>1$, positive definiteness may not imply Hermitian positive definiteness.
\begin{exam}
Let $A$ be a $4$-th order 2-dimensional CPS tensor defined by: $a_{11\bar{1}\bar{1}}=a_{22\bar{2}\bar{2}}=1, a_{11\bar{2}\bar{2}}=a_{22\bar{1}\bar{1}}=2$, $a_{ij\bar{k}\bar{l}}=0$, otherwise. Then for $\mathbf{x}=(x_1,x_2)^T\in \mathbb R^2\setminus \{0\}$, $$f(A)(\mathbf{x})=x_1^4+4x_1^2x_2^2+x_2^4>0.$$ When $\mathbf{x}=(x_1,x_2)^T\in \mathbb C^2\setminus \{0\}$, $$f(A)(\mathbf{x})=|x_1|^4+2x_1^2\overline{x_2}^2+2x_2^2\overline{x_1}^2+|x_2|^4$$ may not be positive. For example, when $\mathbf{x}=(\sqrt{-1},1)^T, f(A)(\mathbf{x})=-2<0$
\end{exam}

\section{$\hat{H}$-eigenvalues}
In this section, we define $\hat{H}$-eigenvalue for $n$-dimensional complex tensors and relate it to Hermitian positive definiteness. As Hermitian positive definiteness is stronger than positive definiteness, it is not sufficient to study the H-eigenvalue in this situation. We introduce the following.
\begin{defn}
		Let $A=(a_{i_1\cdots i_m\bar{j}_1\cdots \bar{j}_m})$ be a $2m$-th order $n$-dimensional complex tensor. We call that $\lambda\in \mathbb C$ is an $\hat{H}$-eigenvalue of $A$ with $\hat{H}$-eigenvector $\mathbf{x}$ if there is a nonzero complex vector $\mathbf{x}=(x_1,\cdots, x_n)^T\in \mathbb C^n\setminus \{0\}$ such that 
$$\sum_{i_2,\cdots, i_m, j_1, \cdots, j_m=1}^n a_{ii_2\cdots i_m\bar{j}_1\cdots \bar{j}_m} x_{i_2}\cdots x_{j_m}\overline{x_{j_1}}\cdots\overline{x_{j_m}}=\lambda \bar{x}_{i}|x_{i}|^{2m-2}, {\forall}\  1\leq i\leq n.$$ 
	\end{defn}
It follows from the definition that if $\mathbf{x}$ is an $\hat{H}$-eigenvector of $\lambda$, then $k\mathbf{x}$ is also an $\hat{H}$-eigenvector of $\lambda$ for any $k\in \mathbb C\setminus \{0\}$. If $\mathbf{x}$ is a real vector, it is at the same time an $H$-eigenvector of $\lambda$. However, $\hat{H}$-eigenvectors may not be real in general and the $\hat{H}$-eigenvalues differ from H-eigenvalues even for Hermitian tensors. For $\hat{H}$-eigenvalues of Hermitian tensors, we have
\begin{lem}
If $A$ is a Hermitian tensor, then any $\hat{H}$-eigenvalue $\lambda$ of $A$ must be real.
\end{lem}
\begin{proof}
Assume that $\mathbf{x}$ is a $\hat{H}$-eigenvector of $\lambda$. Then 
\begin{align*}
		f(A)(\mathbf x)=\sum_{i, i_2,\cdots, i_m, j_1, \cdots, j_m=1}^n a_{ii_2\cdots i_m\bar{j}_1\cdots \bar{j}_m} x_i x_{i_2}\cdots x_{j_m}\overline{x_{j_1}}\cdots\overline{x_{j_m}}=\lambda\sum_{i=1}^n |x_{i}|^{2m}.\end{align*} 
	As $A$ is Hermitian,
	 $$\overline{f(A)(\mathbf x)}=\sum_{i_1,\cdots, i_m, j_1, \cdots, j_m=1}^n \overline{a_{i_1\cdots i_m\bar{j}_1\cdots \bar{j}_m}} \bar{x}_{i_1}\cdots \bar{x}_{j_m}{x_{j_1}}\cdots{x_{j_m}}=f(A)(\mathbf x).$$ So $f(A)(\mathbf x)$ is real and $\lambda=\dfrac{f(A)(\mathbf x)}{\sum_{i=1}^n|x_i|^{2m}}$ is also real.
	\end{proof}
Given a $2m$-th order $n$-dimensional Hermitian tensor $A$, define its conjugate partial symmetrization tensor $S_{A}=(b_{i_1\cdots i_m\bar{j}_1\cdots \bar{j}_m})$ by
\begin{align*}
b_{i_1\cdots i_m\bar{j}_1\cdots \bar{j}_m}=\dfrac{1}{(n!)^2}\sum_{\sigma, \tau\in S_n} a_{i_{\sigma(1)}\cdots i_{\sigma(m)}\bar{j}_{\tau(1)}\cdots \bar{j}_{\tau(m)}} ,
\end{align*}
	where $S_m$ is the permutation group of $m$ elements. Then $S_{A}$ is a CPS tensor. Also
 \begin{align} f(A)(\mathbf x)=\sum_{i_1, i_2,\cdots, i_m, j_1, \cdots, j_m=1}^n a_{i_1i_2\cdots i_m\bar{j}_1\cdots \bar{j}_m} x_{i_1} x_{i_2}\cdots x_{j_m}\overline{x_{j_1}}\cdots\overline{x_{j_m}}\\ \notag
=\sum_{i_1, i_2,\cdots, i_m, j_1, \cdots, j_m=1}^na_{\sigma(i_{1})\cdots \sigma(i_{m})\tau(\bar{j}_{1})\cdots \tau(\bar{j}_{m})}x_{i_1} x_{i_2}\cdots x_{j_m}\overline{x_{j_1}}\cdots\overline{x_{j_m}}=f(S_{A})(\mathbf x). \end{align}
So $A$ is Hermitian positive definite if and only if $S_{A}$ is Hermitian positive definite. We prove the following characterization of Hermitian positive definite tensors.
\begin{prop} \label{3.1}
Let $A$ be a $2m$-th order $n$-dimensional Hermitian tensor. Then there always exists $\hat{H}$-eigenvalues of $S_{A}$. Also, $A$ is Hermitian positive semidefinite (definite) if and only if the $\hat{H}$-eigenvalues of $S_{A}$ are all nonnegative (positive). 
\end{prop}
\begin{proof}
Set $B=\{\mathbf{x}=(x_1,\cdots, x_n)^T\in \mathbb C^n |\sum_{i=1}^{n}|x_i|^{2m}=1\}$. Then $B$ is a bounded closed set in $\mathbb C^n$, hence is compact. As $f(S_{A})(\mathbf{x})=f(A)(\mathbf{x})$ depends continuously on $\mathbf{x}$, $f(S_{A})|_B$ obtains its minimal value at some vector $\mathbf{v}=(v_1, \cdots, v_n)\in B$. Denote $S_{A}=(b_{i_1\cdots i_m\bar{j}_1\cdots \bar{j}_m})$. 
View $x_i, \bar{x}_i$ as independent variables as in complex analysis. According to the Lagrange multiplier method, at $\mathbf{v}$ there is a $\lambda\in \mathbb R$ such that
$$\dfrac{\partial f(S_{A})(\mathbf{x})}{\partial x_i}(\mathbf{v})=\lambda \dfrac{\partial (|x_1|^{2m}+\cdots+|x_n|^{2m}-1)}{\partial x_i}(\mathbf{v}),$$
which gives that 
$$m\sum_{i_2,\cdots, i_m, j_1, \cdots, j_m=1}^nb_{ii_2\cdots i_m\bar{j}_1\cdots \bar{j}_m}v_{i_2}\cdots v_{i_m}\overline{v_{j_1}}\cdots \overline{v_{j_m}}=m\lambda \bar{v}_{i} |v_{i}|^{2m-2},$$
for $1\leq i\leq n$, where the conjugate partial symmetry of $S_{A}$ is used above. So $\mathbf{v}$ is an $\hat{H}$-eigenvector of $S_{A}$ with $\hat{H}$-eigenvalue $\lambda$. Therefore, $\hat{H}$-eigenvalues of $S_{A}$ always exist. Furthermore, if $\lambda\geq 0$ for any $\hat{H}$-eigenvalue $\lambda$ of $S_{A}$, then at the minimum point $\mathbf{v}$ of $f(A)|_B=f(S_{A})|_B$, $$f(A)(\mathbf{v})=f(S_{A})(\mathbf{v})=\lambda(\sum_{i=1}^n|v_{i}|^{2m})\geq 0.$$ 
 Then $f(A)(\mathbf{x})\geq 0$ for any vector $\mathbf{x}\in B$. Now the homogeneity of $f(A)$ means that for any $t\in \mathbb R$,$f(A)(t\mathbf{x})=t^{2m}f(A)(\mathbf{x})$. So we get $f(A)(\mathbf{x})\geq 0$ for any $\mathbf{x}\in \mathbb C^n$.
	
Conversely, assume that $A$ is Hermitian positive semidefinite and $\lambda$ is an $\hat{H}$-eigenvalue of $S_{A}$ with $\hat{H}$-eigenvector $\mathbf{x}$. Then $$0\leq f(A)(\mathbf{x})=f(S_{A})(\mathbf{x})=\lambda(\sum_{i=1}^n|x_{i}|^{2m}).$$ So $\lambda\geq 0$. 

 The case of positive definiteness follows similarly.
\end{proof}
 
\begin{proof}[Proof of Theorem 1.2]  When $A$ is a CPS tensor, $S_{A}=A$. Theorem 1.2 follows from Proposition \ref{3.1}.
\end{proof}
\begin{remk}
In \cite{JLZ}, Jiang-Li-Zhang introduce C-eigenvalues of CPS tensors. It is the Hermitian analogue of Z-eigenvalue of real tensors. They also obtain several properties and applications of C-eigenvalues. As the defining equations of C-eigenvalues are not real homogeneous, the problem to find inclusion sets of C-eigenvalues would be more involved.  
\end{remk}
The following example demonstrates the application of $\hat{H}$-eigenvalues.
\begin{exam}
Let $A$ be a $4$-th order 2-dimensional CPS tensor defined by
 $$a_{11\bar{1}\bar{1}}=a_{22\bar{2}\bar{2}}=1, a_{11\bar{2}\bar{2}}=\sqrt{-1}, a_{22\bar{1}\bar{1}}=-\sqrt{-1}, a_{ij\bar{k}\bar{l}}=0.\ \text{otherwise}$$
The conjuate polynomial given by $A$ is: $$f(A)(\mathbf{x})=|x_1|^4+\sqrt{-1}x_1^2\overline{x_2}^2-\sqrt{-1}x_2^2\overline{x_1}^2+|x_2|^4$$ for $\mathbf{x}=(x_1,x_2)^T\in \mathbb C^2.$ Easy calculation gives that there are three $\hat{H}$-eigenvalues of $A$ which are 0, 1, 2. The
$\hat{H}$-eigenvectors for 0 are $\mathbf{a}=k(e^{\frac{\pi}{4}\sqrt{-1}},1)^T, k\in \mathbb C$; for $1$ are $\mathbf{b}=k(1,0)^T$, and $\mathbf{b'}=k(0,1)^T, k\in \mathbb C$; for $2$ are $\mathbf{c}=k(e^{-\frac{\pi}{4}\sqrt{-1}},1)^T, k\in \mathbb C$. Therefore, by Theorem 1.2, $A$ is a Hermitian positive semidefinite CPS tensor and $f(A)(\mathbf{x})\geq 0$ for any $\mathbf{x}\in \mathbb C^2$.
\end{exam}
In \cite{Q1}, Qi develops the method of characteristic polynomials of symmetric tensor to find the H-eigenvalues of tensors. Among other results, he shows that the maximal total number of H-eigenvalues of a $m$-th order $n$-dimensional real tensor are $n(m-1)^{n-1}$. Using the same method of Qi, we can deduce that for a $2m$-th order $n$-dimensional complex tensor, the total number of its $\hat{H}$-eigenvalues cannot exceed $2n(2m-1)^{2n-1}$. However, the characteristic polynomial of $\hat{H}$-eigenvalues would be more complicated. 

\section{Inclusion sets and Hermitian PSD tensors}
Given a complex tensor $A$, denote the set of all its $\hat{H}$-eigenvalues by $\sigma_h(A)$. If there is a set $K(A)$ such that $\sigma_h(A)\subset K(A)$, then $K(A)$ is called an \textbf{inclusion set} of the $\hat{H}$-eigenvalues. In this section, we discuss several checkable inclusion sets of the $\hat{H}$-eigenvalues. They are essentially generalizations of inclusion sets from H-eigenvalues of real tensors to $\hat{H}$-eigenvalues of complex tensors. As consequences, we give several criterions for Hermitian positive definiteness of a Hermitian tensor.

The first type of inclusion set is a Ger\u{s}gorin type inclusion set first introduced by Qi \cite{Q1} for real tensors. It is a generalization of the Ger\u{s}gorin set of matrix eigenvalues. We show that it is also suitable for $\hat{H}$-eigenvalues of complex tensors.
\begin{defn}\label{d4.1} Assume that $A=(a_{i_1\cdots i_{m}\bar{j}_1\cdots \bar{j}_m})$ is a complex $2m$-th order $n$-dimensional tensor. For each $1\leq i\leq n$, denote $r_i(A)=\sum_{(i_2\cdots i_mj_1\cdots j_{m})\neq (i\cdots i)}|a_{ii_2\cdots i_{m}\bar{j}_1\cdots \bar{j}_m}|$. Define $$K_{ger}(A)=\bigcup _{i=1}^n \{\lambda\in \mathbb C: |\lambda-a_{i\cdots i\bar{i}\cdots\bar{i}}|\leq r_i(A)\}.$$
\end{defn}
We have
\begin {prop} \label{4.01}
$K_{ger}(A)$ is an inclusion set of $\hat{H}$-eigenvalues of $A$.
\end{prop}
\begin{proof}
	Let $\lambda$ be an $\hat{H}$-eigenvalues of $A$ with a $\hat{H}$-eigenvector $\mathbf{x}=(x_1,\cdots,x_n)^T$. Assume that $|x_k|=\max\{|x_1|,\cdots,|x_n|\}>0$. Then
	$$(\lambda-a_{k\cdots k\bar{k}\cdots\bar{k}})\bar{x}_k|x_k|^{2m-2}=\sum_{(i_2\cdots i_mj_1\cdots j_{m})\neq (k\cdots k)} a_{ki_2\cdots i_m\bar{j}_1\cdots \bar{j}_m} x_{i_2}\cdots x_{i_m}\overline{x_{j_1}}\cdots\overline{x_{j_m}}.$$
	So 
	\begin{align*}
	|\lambda-a_{k\cdots k\bar{k}\cdots\bar{k}}|&\leq \sum_{(i_2\cdots i_mj_1\cdots j_{m})\neq (k\cdots k)} |a_{ki_2\cdots i_m\bar{j}_1\cdots \bar{j}_m}|\dfrac{x_{i_2}\cdots x_{i_m}\overline{x_{j_1}}\cdots\overline{x_{j_m}}|}{|x_k|^{2m-1}}\\
	&\leq \sum_{(i_2\cdots i_mj_1\cdots j_{m})\neq (k\cdots k)}|a_{ki_2\cdots i_m\bar{j}_1\cdots \bar{j}_m}|=r_k(A).\end{align*}
	Therefore, $\lambda$ lies in $K_{ger}(A)$ and $K_{ger}(A)$ is an inclusion set of the $\hat{H}$-eigenvalues of $A$.
\end{proof}
Next, we generalize the Li-Li-Kong type inclusion set from real tensors to complex tensors.
In \cite{LLK}, Li-Li-Kong introduce a new inclusion set of H-eigenvalues for real tensors and show that it is tighter than the Ger\u{s}gorin type inclusion set. We introduce an analogue set for complex tensors.
\begin{defn}\label{d4.3}
	Let $A=(a_{i_1\cdots i_{m}\bar{j}_1\cdots \bar{j}_m})$ be a complex $2m$-th order $n$-dimensional tensor. Define
	$$K_{llk}(A)=\bigcup_{i,j=1,i\neq j}^nK_{i,j}(A),$$
	where for $i\neq j$, $$K_{i,j}(A)=\left\{\lambda \in \mathbb C :(|\lambda -a_{i\cdots i\bar{i}\cdots\bar{i}}|-r_{i}^{j}(A))|\lambda-a_{j\cdots j\bar{j}\cdots\bar{j}}|\leq|a_{ij\cdots j\bar{j}\cdots\bar{j}}|r_{j}(A) \right\}$$
and \begin{align*}
 r_{i}^{j}(A)=r_{i}(A)-|a_{ij\cdots j\bar{j}\cdots\bar{j}}|.\end{align*}
\end{defn}
We have
\begin{prop} \label{4.4}
	$\sigma_h(A)\subseteq K_{llk}(A)\subseteq K_{ger}(A)$. Therefore, $K_{llk}(A)$ is an $\hat{H}$-eigenvalues inclusion set which is tighter than $K_{ger}(A)$.
\end{prop}
\begin{proof}
We show that $\sigma_h(A)\subseteq K_{llk}(A)$. Let $\lambda$ be an $\hat{H}$-eigenvalues of $A$ with an $\hat{H}$-eigenvector $\mathbf{x}$. Assume that $|x_i|$ and $|x_j|$ are the largest and the second largest of $|x_1|,\cdots,|x_n|$, namely, $|x_i|=\max \{|x_1|,\cdots,|x_n|\}$ and $|x_j|=\max \{|x_1|,\cdots,|x_{i-1}|,|x_{i+1}|,\cdots,|x_n|\}$. It is possible that $|x_i|= |x_j|$. Then 
\begin{align*}
(\lambda-a_{i\cdots i\bar{i}\cdots\bar{i}})\bar{x}_i|x_i|^{2m-2}&=\sum_{(i_2\cdots i_mj_1\cdots j_{m})\neq (i\cdots i)} a_{ii_2\cdots i_m\bar{j}_1\cdots \bar{j}_m} x_{i_2}\cdots x_{i_m}\overline{x_{j_1}}\cdots\overline{x_{j_m}}\\
&=\sum_{(i_2\cdots i_mj_1\cdots j_{m})\neq (i\cdots i), \neq (j\cdots j)} a_{ii_2\cdots i_m\bar{j}_1\cdots \bar{j}_m} x_{i_2}\cdots x_{i_m}\overline{x_{j_1}}\cdots\overline{x_{j_m}}\\
& \ \ \ \  +a_{ij\cdots j\bar{j}\cdots\bar{j}} \bar{x}_j|x_j|^{2m-2}.\end{align*}
So $|\lambda-a_{i\cdots i\bar{i}\cdots\bar{i}}||x_i|^{2m-1}\leq r_{i}^{j}(A)|x_i|^{2m-1}+a_{ij\cdots j\bar{j}\cdots\bar{j}}|x_j|^{2m-1}$, from which we get 
\begin{align} (|\lambda -a_{i\cdots i\bar{i}\cdots\bar{i}}|-r_{i}^{j}(A))||x_i|^{2m-1}\leq |a_{ij\cdots j\bar{j}\cdots\bar{j}}| |x_j|^{2m-1}. \label{4.1}
\end{align}
On the other side, from the proof of Proposition \ref{4.01}, we have \begin{align} \label{4.2}
|\lambda-a_{j\cdots j\bar{j}\cdots\bar{j}}||x_j|^{2m-1}\leq r_{j}(A)|x_i|^{2m-1}.\end{align}
Multiplying (\ref{4.1}) and (\ref{4.2}) we get $(|\lambda -a_{i...i}|-r_{i}^{j}(A))|\lambda-a_{j...j}|\leq|a_{ij...j}|r_{j}(A)$. So $\lambda\in K_{i,j}(A)$. By our assumption, we get $\sigma_h(A)\subseteq K_{llk}(A)$.

The proof of $K_{llk}(A)\subseteq K_{ger}(A)$ follows similarly with the real tensor case which we omit (see \cite{LLK} or \cite{Q1}).
\end{proof}
We can also generalize the following Li-Li type inclusion set defined in \cite{LL} from real tensors to complex tensors.
Assume that $A=(a_{i_1\cdots i_{m}\bar{j}_1\cdots \bar{j}_m})$ is a complex $2m$-th order $n$-dimensional tensor.
	For $1\leq i\leq n$, denote
	$$\bigtriangleup_{i}=\left\{(i_{2}\cdots i_{m}j_1\cdots j_m): (i_{2}\cdots i_{m}j_1\cdots j_m)\neq (i,...,i)\right\},$$
	$$\bigtriangleup_{i}'=\left\{(i_{2}\cdots i_{m}j_1\cdots j_m) \in \bigtriangleup_{i} : i_{k}\neq i, 2\leq k\leq m, j_k\neq i, 1\leq k\leq m\right\},$$
and
	$$\hat{\bigtriangleup}_{i}=\bigtriangleup_{i}-\bigtriangleup_{i}'.$$ 
Denote $$r'_{i}(A)=\sum_{(i_{2}\cdots i_{m}j_1\cdots j_m)\in \bigtriangleup_{i}'}|a_{i_1\cdots i_{m}\bar{j}_1\cdots \bar{j}_m}|$$ and $$\hat{r}_{i}(A)=\sum_{(i_{2}\cdots i_{m}j_1\cdots j_m)\in \hat{\bigtriangleup}_{i}}|a_{i_1\cdots i_{m}\bar{j}_1\cdots \bar{j}_m}|.$$ 
	Then $$r_{i}(A)=r_{i}^{j}(A)+|a_{ij\cdots j\bar{j}\cdots\bar{j}}|=r'_{i}(A)+\hat{r_{i}}(A).$$ 
So $$r_{i}^{j}(A)=r'_{i}(A)+\hat{r}_{i}(A)-|a_{ij\cdots j\bar{j}\cdots\bar{j}}|.$$
	\begin{defn} \label{d4.5} Let
$$K_{i,j}^{ll}(A)=\left\{\lambda\in \mathbb C :(|\lambda-a_{i\cdots i\bar{i}\cdots\bar{i}}|-\hat{r}_{i}(A))|\lambda-a_{j\cdots j\bar{j}\cdots\bar{j}}|\leq r'_{i}(A)r_{j}(A)\right\}$$
and $$K_{ll}(A)=\bigcup_{i,j=1,i\neq j}^n K_{i,j}^{ll}(A).$$
\end{defn}
We have
\begin{prop}
	$\sigma_h(A)\subseteq K_{ll}(A)\subseteq K_{llk}(A)\subseteq K_{ger}(A).$ Therefore, $K_{ll}(A)$ is an inclusion set for the $\hat{H}$-eigenvalues.
\end{prop}
\begin{proof}
We show that $\sigma_h(A)\subseteq K_{ll}(A)$.
Let $\lambda$ be an $\hat{H}$-eigenvalues of $A$ with an $\hat{H}$-eigenvector $\mathbf{x}$. 
Let $|x_{t}|$and $|x_{s}|$be the largest and the second largest of $|x_{1}|,...|x_{n}|.$ Then$|x_{t}|>0.$ By the definition of $\hat{H}$-eigenvalues, we have 
\begin{align*}
(\lambda-a_{t\cdots t\bar{t}\cdots\bar{t}})\bar{x}_t|x_t|^{2m-2}=\sum_{(i_2\cdots i_mj_1\cdots j_{m})\in \hat{\bigtriangleup}_{t}} a_{ti_2\cdots i_m\bar{j}_1\cdots \bar{j}_m} x_{i_2}\cdots x_{i_m}\tilde{x_{j_1}}\cdots\overline{x_{j_m}}\\
+\sum_{(i_2\cdots i_mj_1\cdots j_{m})\in \bigtriangleup'_{t}}a_{ti_2\cdots i_mj_1\cdots j_m}  x_{i_2}\cdots x_{i_m}\overline{x_{j_1}}\cdots\overline{x_{j_m}}.\end{align*}
This implies that
\begin{align*}
|\lambda-a_{t\cdots t\bar{t}\cdots\bar{t}}||x_{t}|^{2m-1}&\leq\sum_{(i_2\cdots i_mj_1\cdots j_{m})\in \hat{\bigtriangleup}_{t}} |a_{ti_2\cdots i_m\bar{j}_1\cdots \bar{j}_m}||x_{t}|^{2m-1}\\
&+\sum_{(i_2\cdots i_mj_1\cdots j_{m})\in\bigtriangleup'_{t}}|a_{ti_2\cdots i_m\bar{j}_1\cdots \bar{j}_m}||x_{s}|^{2m-1}\\
&=\hat{r}_{t}(A)|x_{t}|^{2m-1}+r'_{t}(A)|x_{s}|^{2m-1}.
\end{align*}
Thus, we have \begin{align}(|\lambda-a_{t\cdots t\bar{t}\cdots\bar{t}}|-\hat{r}_{t}(A))|x_{t}|^{2m-1}\leq r'_{t}(A)|x_{s}|^{2m-1}.\label{4.3}\end{align}
Assume that $|x_{s}|>0.$ From the proof of Proposition \ref{4.01}, we have 
\begin{align} |\lambda-a_{s\cdots s\bar{s}\cdots\bar{s}}||x_{s}|^{2m-1}\leq r_{s}(A)|x_{t}|^{2m-1}.\label{4.4}\end{align}
Since $|x_{t}x_{s}|>0$, multiplying (\ref{4.3}) and (\ref{4.4}),we have
$(|\lambda-a_{t...t}|-\hat{r}_{t}(A))|\lambda-a_{s...s}|\leq r'_{t}r_{s}(A).$
So $\lambda\in K_{t,s}^{ll}(A).$ If $x_{s}=0$, then$|\lambda-a_{t\cdots t\bar{t}\cdots\bar{t}}|-\hat{r}_{t}(A)\leq 0 $ since $|x_{t}|>0.$ This also implies that $\lambda\in K_{t,s}^{ll}(A)\subseteq K_{ll}(A).$ Therefore, $K_{ll}(A)$ is an inclusion set for the $\hat{H}$-eigenvalues.

The proof of $K_{ll}(A)\subseteq K_{llk}(A)$ is the same with the real tensor case which we omit (see \cite{LL}).\end{proof}
The following example of Hermitian matrix shows the above inclusion is strict.
\begin{exam}
Let $A=\begin{pmatrix} 1&\sqrt{-1}&0\\-\sqrt{-1}&2&1\\0&1&3\end{pmatrix}$ be a $(3\times 3)$ Hermitian matrix. Then we have $r_1=1, r_2=2, r_3=1$, and $r^2_1=0, r^3_1=1, r_2^1=1, r_2^3=1,r_3^1=1, r_3^2=0$. Also, $r_1'=1, r_2'=2, r_3'=1, \hat{r}_1=\hat{r}_2=\hat{r}_3=0$.
So we get the following description:
$$ \sigma_h(A)=\{2-\sqrt{3},2,2+\sqrt{3}\}, \ \ \  K_{ger}(A)=\{\lambda\in \mathbb C: |\lambda-2|\leq 2\}$$
 \begin{align*} K_{llk}(A)&=\{\lambda\in \mathbb C:|(\lambda-1)(\lambda-2)|\leq 2\}\cup \{\lambda\in \mathbb C:(|\lambda-2|-1)|\lambda-1|\leq 1\}\\
& \ \ \ \ \ \ \  \cup \{\lambda\in \mathbb C:(|\lambda-2|-1)|\lambda-3|\leq 1\} \cup \{\lambda\in \mathbb C:|(\lambda-3)(\lambda-2)|\leq 2\} \end{align*}
\begin{align*} K_{ll}(A)=&\{\lambda\in \mathbb C:|(\lambda-1)(\lambda-2)|\leq 2\}\cup \{\lambda\in \mathbb C:|(\lambda-3)(\lambda-2)|\leq 2\}\\
&\cup\{\lambda\in \mathbb C:|(\lambda-1)(\lambda-3)|\leq 1\}
\end{align*}
It follows that $\sigma_h(A)\subsetneq K_{ll}(A)\subsetneq K_{llk}(A)\subsetneq K_{ger}(A)$. For example, $2+2\sqrt{-1}\in  K_{ger}(A)\setminus K_{llk}(A)$ and $2+
\dfrac{3\sqrt{-1}}{2}\in K_{llk}(A)\setminus K_{ll}(A)$.
\end{exam}
We remark that there are other inclusion sets for H-eigenvalues of real tensors (see \cite{BWL}\cite{LCL}\cite{LJL}\cite{LZL}\cite{WZC}). They can be generalized to be inclusion sets for $\hat{H}$-eigenvalues of complex tensors similarly.

In the rest of the section, we will make use of the three $\hat{H}$-eigenvalues inclusion sets to give some checkable criterions for the Hermitian positive semi-definiteness (PSD) of Hermitian tensors. We will use the same notations with real tensors so as to be consistent (see \cite{QL}). 
\begin{defn}
	Let $A$ be a $2m$-th order $n$-dimensional Hermitian tensor. We call that $A$ is diagonally dominated if for any $1\leq i\leq n$, $a_{i\cdots i\bar{i}\cdots\bar{i}}\geq r_i(A).$ We call $A$ is strictly diagonally dominated if for any $1\leq i\leq n$, $a_{i\cdots i\bar{i}\cdots\bar{i}}> r_i(A).$ \end{defn}
	From Proposition 4.2, we have 
	\begin{cor} \label{4.5}
	If the conjugate partial symmetrization $S_A$ of $A$ is (strictly) diagonally dominated, then $A$ is Hermitian (PD) PSD.
	\end{cor}
	\begin{proof}
	Since $A$ is Hermitian (PD) PSD if and only if $S_A$ is Hermitian (PD) PSD. For simplicity, we assume that $A=S_A$. Let $\lambda$ be a $\hat{H}$-eigenvalue of $S_A$. Then $\lambda\in K_{ger}(S_A)$. So for some $i$, $|\lambda-a_{i\cdots i\bar{i}\cdots\bar{i}}|\leq r_i(S_A)$. Then $\lambda\geq a_{i\cdots i\bar{i}\cdots\bar{i}}-r_i(S_A)\geq 0$. Therefore, $S_A$ is Hermitian PSD by Proposition \ref{3.1} which also gives the Hermitian PSD of $A$ as $f(A)=f(S_A)$. The case of strictly diagonally dominated tensors follows similarly. 
	\end{proof}
	Next, we define the concepts of LLK tensor. 
\begin{defn}
	Let $A$ be a $2m$-th order $n$-dimensional Hermitian tensor. We call that $A$ an LLK tensor if all the diagonal entries $a_{i\cdots i\bar{i}\cdots\bar{i}}\geq0$, and for $1\leq i,j\leq n, i\neq j,$
	$$(a_{i\cdots i\bar{i}\cdots\bar{i}}-r_{i}^{j}(A))a_{j\cdots j}\geq r_{j}(A)|a_{ij\cdots j\bar{j}\cdots \bar{j}}|.$$
If furthermore, $a_{i\cdots i\bar{i}\cdots\bar{i}}>0$ for any $1\leq i\leq n$ and strict inequality holds in the above inequality, then A is called a strict LLK tensor.
\end{defn}
We have 
	\begin{cor}\label{4.6}
	If the conjugate partial symmetrization $S_A$ of $A$ is a (strictly) LLK tensor, then $A$ is Hermitian (PD) PSD.
	\end{cor}
	\begin{proof}
	As before, we assume that $A=S_A$ and prove the case of Hermitian PSD. By Proposition \ref{3.1}, it suffices to show that its $\hat{H}$-eigenvalues are nonnegative. Assuming by contradiction, there is a $\hat{H}$-eigenvalue $\lambda$ of $S_A$ with $\lambda< 0$. Then by Proposition \ref{4.4}, $\lambda\in K_{i,j}(S_A)$ for some $i\neq j$. Therefore, $(|\lambda -a_{i\cdots i\bar{i}\cdots\bar{i}}|-r_{i}^{j}(S_A))|\lambda-a_{j\cdots j\bar{j}\cdots\bar{j}}|\leq |a_{ij\cdots j\bar{j}\cdots \bar{j}}|r_{j}(S_A)$ holds. On the other side, since $\lambda< 0$, we have $|\lambda -a_{i\cdots i\bar{i}\cdots\bar{i}}|-r_{i}^{j}(S_A)>a_{i\cdots i\bar{i}\cdots\bar{i}}-r_{i}^{j}(S_A)$ and $|\lambda-a_{j\cdots j\bar{j}\cdots\bar{j}}|> a_{j\cdots j\bar{j}\cdots\bar{j}}\geq 0$. So $$(|\lambda -a_{i\cdots i\bar{i}\cdots\bar{i}}|-r_{i}^{j}(S_A))|\lambda-a_{j\cdots j\bar{j}\cdots\bar{j}}|> (a_{i\cdots i\bar{i}\cdots\bar{i}}-r_{i}^{j}(S_A))a_{j\cdots j\bar{j}\cdots\bar{j}}\geq r_{j}(S_A)|a_{ij\cdots j\bar{j}\cdots \bar{j}}|$$ where the last inequality follows from the definition. This gives a contradiction. So $A$ is Hermitian PSD.
	\end{proof}
	Finally, we define the following kind of tensor.
\begin{defn}
$A$ is called an LL tensor if al the diagonal entries $a_{i\cdots i\bar{i}\cdots\bar{i}}\geq0,$ and for $1\leq i,j\leq n, i\neq j,$
	\begin{align} (a_{i\cdots i\bar{i}\cdots\bar{i}}-\hat{r}_{i}(A))a_{j\cdots j\bar{j}\cdots\bar{j}}\geq r'_{i}(A)r_{j}(A).\label{4.13} \end{align}
	If furthermore, $a_{i\cdots i\bar{i}\cdots\bar{i}}>0$ for $1\leq i\leq n$ and strict inequality holds in the above inequality, then A is called a strict LL tensor.
\end{defn}
With the same proof as in Corollary \ref{4.6}, the following still holds. 
	\begin{cor}
	If the conjugate partial symmetrization $S_A$ of $A$ is an (strict) LL tensor, then $A$ is Hermitian (PD) PSD.
	\end{cor}
	
The relations between the three Hermitian PSD tensors are given by the following.
\begin{prop}
Let $A$ be a $2m$-th order $n$-dimensional Hermitian tensor. If A is (strictly) diagonally dominated, then A is an (strict) LLK tensor. If A is an (strict) LLK tensor, then A is an (strict) LL tensor.
	\end{prop}
\begin{proof}
Suppose that A is a diagonally dominated tensor. Then for $1\leq i\leq n, a_{i\cdots i\bar{i}\cdots\bar{i}}\geq 0.$ Furthermore, for $i\neq j,$  we have
\begin{align} a_{j\cdots j\bar{j}\cdots\bar{j}}\geq r_{j}(A) \label{4.10}\end{align} and
$a_{i\cdots i\bar{i}\cdots\bar{i}}\geq r_{i}(A)=r_{i}^{j}(A)+|a_{ij\cdots j\bar{j}\cdots \bar{j}}|,$
i.e., \begin{align}  a_{i\cdots i\bar{i}\cdots\bar{i}}-r_{i}^{j}(A)\geq |a_{ij\cdots j\bar{j}\cdots \bar{j}}|.\label{4.11}\end{align}
Multiplying (\ref{4.10}) and (\ref{4.11}), we directly get that $A$ is an LLK tensor.

Next, suppose that A is an LLK tensor. Then either A is diagonally dominated or A is not diagonally dominated.
If A is diagonally dominated, for all $1\leq i, j\leq n$, we have
$a_{j\cdots j\bar{j}\cdots\bar{j}}\geq r_{j}(A)$ and
$a_{i\cdots i\bar{i}\cdots\bar{i}}-\hat{r}_{i}(A)\geq r'_{i}(A).$
Multiplying them we get that $$(a_{i\cdots i\bar{i}\cdots\bar{i}}-\hat{r}_{i}(A))a_{j\cdots j\bar{j}\cdots\bar{j}}\geq r'_{i}(A)r_{j}(A).$$
So $A$ is an LL tensor.

Assume that A is not a diagonally dominated tensor but an LLK tensor. Then there is a $1\leq i_{0}\leq n$ such that
$0\leq a_{i_{0}\cdots i_{0}\bar{i}_0\cdots \bar{i}_0}<r_{i_{0}}(A).$
As $$(a_{i\cdots i\bar{i}\cdots\bar{i}}-r_{i}^{j}(A))a_{j\cdots j\bar{j}\cdots\bar{j}}\geq r_{j}(A)|a_{ij\cdots j\bar{j}\cdots \bar{j}}|$$ for all $i\neq j$ by definition, we have
$a_{i\cdots i\bar{i}\cdots\bar{i}}-r_{i}^{j}(A)\geq |a_{ij\cdots j\bar{j}\cdots \bar{j}}|$, for $i\neq i_0$. So $$a_{i\cdots i\bar{i}\cdots\bar{i}}\geq r_{i}^{j}(A)+|a_{ij\cdots j\bar{j}\cdots \bar{j}}|=r_{i}(A)=\hat{r}_{i}(A)+r'_{i}(A).$$ 
Thus, we get that $(a_{i\cdots i\bar{i}\cdots\bar{i}}-\hat{r}_{i}(A))a_{j\cdots jj}\geq r'_{i}(A)r_{j}(A)$ for all $i,j\neq i_{0}$ for $i, j\neq i_0$. So we only need to deal with the pairs $(i_0, j)$ and $(i, i_0)$ for $i, j\neq i_0$.
For $j\neq i_{0}$, we have
$$(a_{i_{0}\cdots i_{0}\bar{i}_0\cdots \bar{i}_0}-r_{i_{0}}^{j}(A))a_{j\cdots j\bar{j}\cdots \bar{j}}\geq r_{j}(A)|a_{i_{0}j\cdots j\bar{j}\cdots \bar{j}}|. $$
As $0\leq a_{i_{0}\cdots i_{0}\bar{i}_0\cdots \bar{i}_0}<r_{i_{0}}(A),$ we get $a_{i_{0}\cdots i_{0}\bar{i}_0\cdots \bar{i}_0}-r_{i_{0}}^{j}(A)<|a_{i_{0}j\cdots j\bar{j}\cdots \bar{j}}|, a_{i_{0}\cdots i_{0}\bar{i}_0\cdots \bar{i}_0}-\hat{r}_{i_{0}}(A)<r'_{i_{0}}(A).$ If $r_{j}(A)>0,r'_{i_{0}}(A)>0$ then
$$\dfrac{a_{i_{0}\cdots i_{0}\bar{i}_0\cdots \bar{i}_0}-\hat{r}_{i_{0}}(A)}{r'_{i_0}(A)}\dfrac{a_{j\cdots j\bar{j}\cdots \bar{j}}}{r_{j}(A)}\geq \dfrac{a_{i_{0}\cdots i_{0}\bar{i}_0\cdots \bar{i}_0}-r_{i_{0}}^{j}(A)}{|a_{i_{0}j\cdots j\bar{j}\cdots \bar{j}}|}\dfrac{a_{j\cdots j\bar{j}\cdots \bar{j}}}{r_{j}(A)}\geq 1,$$
which implies that (\ref{4.13}) holds for$(i_0, j)$. If $r_{j}(A)=0$ or $r'_{i_{0}}=0$, we have
$$(a_{i_{0}\cdots i_{0}\bar{i}_0\cdots \bar{i}_0}-\hat{r}_{i_{0}(A)})a_{j\cdots j\bar{j}\cdots \bar{j}}\geq 0=r'_{i_{0}}(A)r_{j}(A).$$
So (\ref{4.13}) holds for $i=i_{0}$ and j.
Assume that $i\neq i_{0}.$ By definition we have 
$$(a_{i\cdots i\bar{i}\cdots \bar{i}}-r_{i}^{i_0}(A))a_{i_0\cdots i_0\bar{i}_0\cdots\bar{i}_0}\geq r_{i_0}(A)|a_{ii_0\cdots i_0\bar{i}_0\cdots\bar{i}_0}|.$$
Then $a_{i\cdots i\bar{i}\cdots\bar{i}}\geq r_{i}(A)$ for $i\neq i_{0}$ which gives that $a_{i\cdots i\bar{i}\cdots\bar{i}}-\hat{r}_{i}(A)\geq r'_{i}(A)\geq0.$ If $r_{i_0}(A)>0,r'_{i}(A)>0$, then
$$\dfrac{a_{i\cdots i\bar{i}\cdots\bar{i}}-\hat{r}_{i}(A)}{r'_{i}(A)}\dfrac{a_{i_{0}\cdots i_{0}\bar{i}_0\cdots \bar{i}_0}}{r_{i_{0}}(A)}\geq \frac{a_{i\cdots i\bar{i}\cdots \bar{i}}-r_{i}^{i_0}(A)}{|a_{ii_0\cdots i_0\bar{i}_0\cdots\bar{i}_0}|}\frac{a_{i_0\cdots i_0\bar{i}_0\cdots \bar{i}_0}}{r_{i_0}(A)}\geq 1$$
which implies that (\ref{4.13}) holds for $(i, i_{0}).$ If $r_{i_0}(A)=0$ or $r'_i(A)=0$, then 
$$(a_{i\cdots i\bar{i}\cdots\bar{i}}-\hat{r_{i}}(A))a_{i_{0}\cdots i_{0}\bar{i}_0\cdots \bar{i}_0}\geq0=r'_{i}(A)r_{i_{0}}(A).$$
So(\ref{4.13}) still holds for $(i, i_{0})$. 
Thus, (\ref{4.13}) always holds for $i\neq j$ and $A$ is an LL tensor.
The proofs for strict tensors are similar. 
\end{proof}	
The following examples demonstrate the differences of the three types of Hermitian tensors.
\begin{exam}
Let $$A=\begin{pmatrix} 3 & \sqrt{-1} &0\\-\sqrt{-1}&3&1\\
0&1&3\end{pmatrix}, B=\begin{pmatrix} 3 & \sqrt{-1} &0\\-\sqrt{-1}&\frac{3}{2}&1\\
0&1&3\end{pmatrix},
C=\begin{pmatrix} \frac{1}{2} & \sqrt{-1} &0\\-\sqrt{-1}&5&1\\
0&1&3\end{pmatrix}.$$ Then $A, B, C$ have the same $r_1=1, r_2=2, r_3=1$; $r^2_1=0, r^3_1=1, r_2^1=1, r_2^3=1,r_3^1=1, r_3^2=0$ and $r_1'=1, r_2'=2, r_3'=1, \hat{r}_1=\hat{r}_2=\hat{r}_3=0$. From the definitions, we directly get that $A$ is strictly diagonally dominated. $B$ is strictly LLK but not diagonally dominated. $C$ is strictly LL but not LLK. Off course, they are all Hermitian PD.
\end{exam}
Let $\mathcal A=(a_{i_1\cdots i_m\bar{j}_1\cdots \bar{j}_m}), 1\leq i_1,\cdots, i_m,j_1,\cdots,j_m\leq n$ be a $2m$-th order $n$-dimensional Hermitian tensor. For $1\leq s\leq n$, denote the $s$-dimensional Hermitian sub-tensor $$B=(a_{i_1\cdots i_m\bar{j}_1\cdots \bar{j}_m}), 1\leq i_1,\cdots, i_m,j_1,\cdots,j_m\leq s$$ and $n-s$-dimensional Hermitian sub-tensor $$C=(a_{i_1\cdots i_m\bar{j}_1\cdots \bar{j}_m}), s+1\leq i_1,\cdots, i_m,j_1,\cdots,j_m\leq n.$$ Let $K=max\{|a_{i_1\cdots i_m\bar{j}_1\cdots \bar{j}_m}|: a_{i_1\cdots i_m\bar{j}_1\cdots \bar{j}_m}\in A\setminus (B\cup C)\}$. Namely, $K$ is the maximum modulus of entries of $A$ which has at least an index less or equal to $s$ and at least an index greater or equal to $s+1$. 
We have the following approximation criterion of Hermitian PSD Hermitian tensors which would be useful in next section.
\begin{prop}\label{p4.16}
If $B$ and $C$ are Hermitian PD with $\hat{H}$-eigenvalue $\lambda_{S_B}\geq K_1>0, \lambda_{S_C}\geq K_2>0$. Then there is a constant $C>0$ depending on $\dfrac{K}{K_1}$ such that if $\dfrac{K_2}{K}\geq C$, then $A$ is Hermitian PD.
\end{prop}
\begin{proof}
By the proof of Proposition 3.3, for $\mathbf{x}=(x_1,\cdots, x_n)^T$, we get that \begin{align*}
\sum_{i_1,\cdots, i_m, j_1, \cdots, j_m=1}^s a_{i_1\cdots i_m\bar{j}_1\cdots \bar{j}_m} x_{i_1} x_{i_2}\cdots x_{j_m}\overline{x_{j_1}}\cdots\overline{x_{j_m}}\geq K_1(|x_1|^{2m}+\cdots+|x_s|^{2m})\\
\sum_{i_1, i_2,\cdots, i_m, j_1, \cdots, j_m=s+1}^n a_{i_1\cdots i_m\bar{j}_1\cdots \bar{j}_m} x_{i_1} x_{i_2}\cdots x_{j_m}\overline{x_{j_1}}\cdots\overline{x_{j_m}}\geq K_2(|x_{s+1}|^{2m}+\cdots+|x_n|^{2m})
\end{align*}
The arithmetic-geometric mean inequality gives $$|x_{i_1}\cdots x_{i_m}\overline{x_{j_1}}\cdots\overline{x_{j_m}}|\leq \dfrac{1}{2m}(a_{i_1}|x_{i_1}|^{2m}+\cdots+a_{i_m}|x_{i_m}|^{2m}+a_{j_1}|x_{j_1}|^{2m}+a_{j_m}|x_{j_m}|^{2m}),$$ where $a_i>0$ satisfies that $a_{i_1}\cdots a_{j_m}=1$. Let $N=(2m)^{n}-(2m)^s-(2m)^{n-s}$. If there is at least one index among $\{i_1,\cdots, i_m, j_1, \cdots, j_m\}$ less than $s+1$ and at least an index among it greater than $s$, we choose $a_i<\dfrac{K_1}{NK}$ if $i\leq s$ and $a_j=c$ for $j\geq s+1$. By $a_{i_1}\cdots a_{j_m}=1$, for $j\geq s+1$, $a_j\leq \max\{1,(\frac{NK}{K_1})^{2m-1}\}$. For this $\{i_1,\cdots, i_m, j_1, \cdots, j_m\}$, we have $$|x_{i_1}\cdots x_{i_m}\overline{x_{j_1}}\cdots\overline{x_{j_m}}|< \dfrac{K_1}{NK}(|x_{1}|^{2m}+\cdots+|x_{s}|^{2m})+\max\{1,(\frac{NK}{K_1})^{2m-1}\}(|x_{s+1}|^{2m}+\cdots+|x_{n}^{2m}|).$$
Then 
\begin{align*}
f(A)(\mathbf{x})&=\sum_{i_1,\cdots, i_m, j_1, \cdots, j_m=1}^n a_{i_1\cdots i_m\bar{j}_1\cdots \bar{j}_m} x_{i_1} x_{i_2}\cdots x_{j_m}\overline{x_{j_1}}\cdots\overline{x_{j_m}}\\
&> K_1(|x_1|^{2m}+\cdots+|x_s|^{2m})+K_2(|x_{s+1}|^{2m}+\cdots+|x_n|^{2m})-NK(\dfrac{K_1}{NK}(|x_{1}|^{2m}+\cdots+|x_{s}|^{2m})\\
& \ \ \ \ \ +\max\{1,(\frac{NK}{K_1})^{2m-1}\}(|x_{s+1}|^{2m}+\cdots+|x_{n}^{2m}|))\\
&\geq (K_2-\max\{NK,\dfrac{(NK)^{2m}}{(K_1)^{2m-1}}\})(|x_{s+1}|^{2m}+\cdots+|x_{n}^{2m}|)
\end{align*}
So if $C=\max\{N,N^{2m}(\dfrac{K}{K_1})^{2m-1}\}$ and $\dfrac{K_2}{K}\geq C$, then $f(A)(\mathbf{x})> 0$ for nonzero $\mathbf{x}$ and $A$ is Hermitian PD.
 \end{proof}

\section{Relation to holomorphic sectional curvature}
In this section, we relate the concept of Hermitian tensors to the holomorphic sectional curvature on a Hermitian manifold. Holomorphic sectional curvature is one of the most important invariants on a Hermitian manifold. Many deep results in algebraic geometry and complex analysis such as the ampleness of the canonical bundle, Schwarz lemma, Hartogs extension theorem have close relations with the holomorphic sectional curvature.

Let $(M,J, g)$ be an $n$-dimensional Hermitian manifold. Namely, $M$ is a smooth manifold, $J$ is an integrable almost complex structure on $M$ and $g$ is a Riemannian metric on $M$ satisfying $g(JX,JY)=g(X,Y)$ for any tangent vectors $X,Y$. Assume that $(U, \varphi:U\rightarrow \mathbb C^n)$ is a local holomorphic coordinate chart with $(z_1, \cdots, z_n)$ being the coordinates. Let $\omega$ be the fundamental form of $g$ which is defined by $\omega(X,Y)=g(JX,Y)$. $g$ is called a K\"ahler metric if $d\omega=0$. Denote $T^{1,0}(M)$ to be the $\sqrt{-1}$ eigenbundle of $J$ on $TM\otimes \mathbb C$ with $\{\frac{\partial}{\partial z_1}, \cdots, \frac{\partial}{\partial z_n}\}$ being a local basis of $T^{1,0}(U)$. Let $g_{i\bar{j}}=g(\frac{\partial}{\partial z_i}, \frac{\partial}{\partial \bar{z}_j})$. Then $(g_{i\bar{j}})$ is a positive definite Hermitian matrix. On $U$, $\omega=\sqrt{-1}\sum_{i,j=1}^ng_{i\bar{j}}dz_i\wedge d\bar{z}_j$. 

A Hermitian connection on a Hermitian manifold is an affine connection $\nabla$ on $M$ satisfying $\nabla J=\nabla g=0$. There are plenty of Hermitian connections on a Hermitian manifold. If the $(1,1)$ part of the torsion of $\nabla$ vanishes, then $\nabla$ is unique and is called the Chern connection  $\nabla^C$ (\cite{Z}). It is known that $g$ is a K\"ahler metric if and only if $\nabla^C$ coincides with the Levi-Civita connection, which is the unique torsion free connection preserving $g$. For any Hermitian connection $\nabla$, the curvature tensor is a $(4,0)$ tensor $R$ on $M$ which is defined by 
$$R(X,Y,Z,W)=g(\nabla_X\nabla_YZ-\nabla_Y\nabla_XZ-\nabla_{[X,Y]}Z, W),$$ where $X,Y,Z,W\in TM\otimes \mathbb C$. From the definition of $R$ and $\nabla g=0$, the following hold
\begin{align}
 R(X,Y,Z,W)&=-R(Y,X,Z,W) \notag \\
R(X,Y,Z,W)&=-R(X,Y,W,Z)\notag \\
\overline{R(X,Y,Z,W)}&=R(\bar{X},\bar{Y},\bar{Z},\bar{W}).\label{5.1}
\end{align} 
The holomorphic sectional curvature $H$ at a point $p\in M$ is a function on the $T_p^{1,0}(M)\setminus \{0\}$ which is defined by $$H(v)=\dfrac{R(v,\bar{v},v,\bar{v})}{|v|_g^4},$$ where $v\in T_p^{1,0}(M)\setminus \{0\}$. We call that $\nabla$ has positive (semipositive, negative, seminegative) holomorphic sectional curvature if $H(v)>0 (\geq 0, <0, \leq 0)$ for any $v\in T_p^{1,0}(M)\setminus \{0\}$ and any $p\in M$.

Assume that $\{e_1,\cdots, e_n\}$ is a local frame of $T^{1,0}(M)$ around $p$. We can define a $4$-th order $n$-dimensional complex tensor from the coefficients of the curvature tensor $R$ as follows, using the convention in section 2 .
\begin{defn}
For the curvature tensor $R$ of a Hermitian connections $\nabla$, define $A_R=(a_{ij\bar{k}\bar{l}}), 1\leq i,j,k,l\leq n$ with $a_{ij\bar{k}\bar{l}}=R(e_i,\bar{e}_k,e_j,\bar{e}_l)$. 
\end{defn}

By (\ref{5.1}), $$\overline{R(e_i,\bar{e}_k,e_j,\bar{e}_l)}=R(\bar{e}_i,e_k,\bar{e}_j,e_l)=R(e_k,\bar{e}_i,e_l,\bar{e}_j).$$ So $\overline{a_{ij\bar{k}\bar{l}}}=a_{kl\bar{i}\bar{j}}$. Therefore, $A_R$ is a Hermitian tensor. If $g$ is a K\"ahler metric with $\nabla$ being the Levi-Civita connection, or if $\nabla$ satisfies K\"ahler-like condition (\cite{YZ}), then the curvature tensor satisfies $R(e_i,\bar{e}_k,e_j,\bar{e}_l)=R(e_j,\bar{e}_k,e_i,\bar{e}_l)=R(e_i,\bar{e}_l,e_j,\bar{e}_k)$. It gives $$a_{ij\bar{k}\bar{l}}=a_{ji\bar{k}\bar{l}}=a_{ij\bar{l}\bar{k}}=a_{ji\bar{l}\bar{k}}.$$ So $A_R$ is a CPS tensor in this case.

Let $v=\sum_{i=1}^n v^ie_i \in T_p^{1,0}(M)\setminus \{0\}$ and $g_{i\bar{j}}=g(e_i,\bar{e}_j)$, then 
$$H(v)=\dfrac{R(v,\bar{v},v,\bar{v})}{|v|_g^4}=\dfrac{\sum_{i,j,k,l=1}^nR(e_i,\bar{e}_k,e_j,\bar{e}_l)v^i\bar{v}^kv^j\bar{v}^l}{(\sum_{i,j=1}^n g_{i\bar{j}}v^i\bar{v}^j)^2}=\dfrac{\sum_{i,j,k,l=1}^na_{ij\bar{k}\bar{l}}v^iv^j\bar{v}^k\bar{v}^l}{(\sum_{i,j=1}^n g_{i\bar{j}}v^i\bar{v}^j)^2}.$$
So the holomorphic sectional curvature $H$ is positive (semipositive) at $p$ if and only if $A_R$ is Hermitian positive definite (semidefinite) at $p$.
As a result of Theorem 1.2, we have 
\begin{cor}
On a Hermitian manifold $(M,J,g)$, the holomorphic sectional curvature of a Hermitian connection $\nabla$ at $p\in M$  is positive (semipositive) if and only if the $\hat{H}$-eigenvalues of $S_{A_R}$ are positive (nonnegative) at $p$.
\end{cor}
We now apply the methods in section 4 to study Hermitian manifolds with positive holomorphic sectional curvature. In \cite{AHZ}, Alvarez-Heier-Zheng prove that the projectivization of any holomorphic vector bundle over a compact K\"ahler manifold with positive holomorphic sectional curvature admits a K\"ahler metric with positive holomorphic sectional curvature. Specifically, let $(M,g)$ be a compact K\"ahler manifold and $\pi: (E,h)\to (M,g)$ be a holomorphic Hermitian vector bundle. They define a K\"ahler metric on the projectivization bundle $P(E)$ as: $$\omega_G=\lambda \pi^*(\omega_g)+\sqrt{-1}\partial\bar{\partial}\log h(v,\bar{v}).$$ 
They compute the curvature tensor of $\omega_G$ and show that $\omega_G$ has positive holomorphic sectional curvature if $\lambda$ is sufficiently large. A key step during the proof is to show when $\lambda$ is sufficiently large, the following tensor $G$ is Hermitian negative definite:
\begin{align*} G_{i\bar{j},k\bar{l}}&=\lambda g_{i\bar{j},k\bar{l}}+h_{v\bar{v},i\bar{j}k\bar{l}}-h_{v\bar{v},i\bar{j}}h_{v\bar{v},k\bar{l}}-h_{v\bar{v},i\bar{l}}h_{v\bar{v},k\bar{j}},\\
G_{i\bar{j},k\bar{\beta}}&=h_{v\bar{\beta},i\bar{j}k},\ \  G_{i\bar{j},\alpha\bar{\beta}}=h_{\alpha\bar{\beta},i\bar{j}}-h_{\alpha\bar{\beta}}h_{v\bar{v},i\bar{j}},\\
 G_{\alpha\bar{j},\gamma\bar{l}}&=0,\ \  G_{\alpha\beta,\gamma\bar{j}}=0,\ \ \
G_{\alpha\bar{\beta},\gamma\bar{\delta}}=-h_{\alpha\bar{\beta}}h_{\gamma\bar{\delta}}-h_{\alpha\bar{\delta}}h_{\gamma\bar{\beta}}
\end{align*}
where $1\leq i,j,k,l\leq n, n+1\leq \alpha, \beta, \gamma,\delta\leq n+r$, $(h_{\alpha\bar{\beta}})$ is positive definite and  $(-g_{i\bar{j},k\bar{l}})$ is a Hermitian PD tensor. Denoting $I_s=(I_{i\bar{j}k\bar{l}})$ to be the $s$-dimensional diagonal identity tensor defined by $I_{i\bar{j}k\bar{l}}=1$, if $i=j=k=l$ and $I_{i\bar{j}k\bar{l}}=0 $, otherwise. Since $(-g_{i\bar{j},k\bar{l}})$ and $(h_{\alpha\bar{\beta}}h_{\gamma\bar{\delta}}+h_{\alpha\bar{\delta}}h_{\gamma\bar{\beta}})$ are both Hermitian PD, there are $a, b>0$ such that $(-g_{i\bar{j},k\bar{l}})\geq a I_n$ and $(h_{\alpha\bar{\beta}}h_{\gamma\bar{\delta}}+h_{\alpha\bar{\delta}}h_{\gamma\bar{\beta}})\geq bI_r$. Here $\geq$ means that the difference of the two tensors are Hermitian PSD. Replacing $(-g_{i\bar{j},k\bar{l}})$ by $a I_n$ and replacing $(h_{\alpha\bar{\beta}}h_{\gamma\bar{\delta}}+h_{\alpha\bar{\delta}}h_{\gamma\bar{\beta}})$ by $bI_{r}$, we get a tensor $G'$ with $G\leq G'$. It follows easily that when $\lambda$ is large enough, $-S_{G'}$ is diagonally dominated hence $-G'$ is Hermitian PD by Corollary 4.9. So $G$ is Hermitian negative definite, which gives a new proof of the algebraic step in \cite{AHZ}.

 In \cite{CH}, Chaturvedi-Heier prove that for a compact holomorphic fibration over a Hermitian manifold with positive holomorphic sectional curvature, if the fibers have a smooth family of Hermitian metrics with positive holomorphic sectional curvature,  
then the total space admits Hermitian metrics of positive holomorphic sectional curvature. It can be compared to the the result of Cheung \cite{C} in the negative curvature case. A key step during their proof is the following lemma (see also \cite{C}).
\begin{lem}[Cheung, Chaturvedi-Heier] \label{l5.3} Let $R_{i\bar{j}k\bar{l}}, 1\leq i,j,k,l\leq n$ be the components of curvature tensor. Suppose the following is true at $p\in M$, for some positive constants $K,K_1,K_2$ and $s<n$:
$$\sum_{i,j,k,l=1}^s R_{i\bar{j}k\bar{l}}\xi^i\bar{\xi}^j\xi^k\bar{\xi}^l\geq K_1\sum_{i,j=1}^s |\xi^i|^2|\xi^j|^2,$$
$$\sum_{\alpha,\beta,\gamma,\delta=s+1}^n R_{\alpha\bar{\beta}\gamma\bar{\delta}}\xi^\alpha\bar{\xi}^\beta\xi^\gamma\bar{\xi}^\delta\geq K_2\sum_{\alpha,\beta=s+1}^n |\xi^\alpha|^2|\xi^\beta|^2,$$
and $|R_{i\bar{j}k\bar{l}}|<K$ whenever $\min(i,j,k,l)\leq s$ and $\max(i,j,k,l)>s$. Then there exists a constant $C$ depending on $\dfrac{K_1}{K}$ such that if $\dfrac{K_2}{K}\geq C$, then $R$ is Hermitian PSD.
\end{lem} 
Since $(|\xi^1|^4+\cdots+|\xi^s|^4)\leq \sum_{i,j=1}^s |\xi^i|^2|\xi^j|^2\leq s(|\xi^1|^4+\cdots+|\xi^s|^4)$ and $(|\xi^{s+1}|^4+\cdots+|\xi^n|^4)\leq \sum_{\alpha,\beta=s+1}^n |\xi^\alpha|^2|\xi^\beta|^2\leq (n-s)(|\xi^{s+1}|^4+\cdots+|\xi^n|^4)$. It is clear that Lemma \ref{l5.3} is a special case of Proposition \ref{p4.16}.\\

\end{document}